\documentclass[12pt]{article}
\usepackage{amsthm}
\usepackage{amssymb}
\usepackage{amsmath}

\newtheorem{thm}{Theorem}[section]
\newtheorem{crl}[thm]{Corollary}

\newtheorem{lmm}[thm]{Lemma}

\newtheorem{prp}[thm]{Proposition}

\theoremstyle{definition}
\newtheorem{dfn}[thm]{Definition}
\theoremstyle{remark}
\newtheorem*{rem}{Remark}%%% * means no numbering
\newcommand{\ai}{\operatorname{Ai}}
\newcommand{\ram}{\operatorname{A}}
\newcommand{\res}{\operatorname{Res}}

\begin{document}
\title{The Stokes phenomenon for the Ramanujan's  $q$-difference equation and its higher order extension}
\author{Takeshi MORITA\thanks{Graduate School of Information Science and Technology, Osaka University, 
1-1  Machikaneyama-machi, Toyonaka, 560-0043, Japan.} }
\date{}
\maketitle
%%%%%%%%%%%%%%%%%%%%%%abstract%%%%%%%%%%%%%%%%%%%%%%%%%%%%%%%%%%%%%%%%%%
\begin{abstract}      %optional
We show connection formulae of local solutions of the Ramanujan equation between the origin and the infinity. These solutions are given by the Ramanujan function, the $q$-Airy function and the divergent basic hypergeometric series ${}_2\varphi_0(0,0;-;q,x)$. We use two different $q$-Borel-Laplace resummation methods to obtain our connection formulae. We also introduce the $q$-Borel-Laplace transformation of level $r-1$, which are higher order extension of these transformations. These methods are useful to obtain an asymptotic formula of a divergent series ${}_r\varphi_0(0,0,\dots ,0;-;q,x)$.
\end{abstract}

\section{Introduction}
In this papar, we deal with the Ramanujan equation
\begin{equation}
qxu(q^2x)-u(qx)+u(x)=0,\qquad \forall x\in\mathbb{C}^*.
\label{ramanujaneq}
\end{equation}
We assume that $0<|q|<1$. The Ramanujan equation \eqref{ramanujaneq} has solutions 
\begin{align}
u_1(x)&={}_0\varphi_1(-;0;q,-qx), \\
u_2(x)&=\theta_q(x){}_2\varphi_0\left(0,0;-;q,-\frac{x}{q}\right)\label{secondram}
\end{align}
around the origin. 
The equation \eqref{ramanujaneq} also has a fundamental system of solutions around the infinity \cite{IZ}:
\begin{align}
v_1(x)&=\frac{\theta_q(x)}{\theta_{q^2}(x)}{}_1\varphi_1\left(0;q;q^2,\frac{q^2}{x}\right)\\
v_2(x)&=\frac{q}{q-1}\frac{\theta_q(x/q)}{\theta_{q^2}(x/q)}\frac{1}{x}{}_1\varphi_1\left(0;q^3;q^2,\frac{q^3}{x}\right).
\end{align}
Here, the function ${}_r\varphi_s(a_1,\dots ,a_r;b_1,\dots ,b_s;q,x)$ is the basic hypergeometric series with the base $q$ \cite{GR}:
\begin{align*}
{}_r\varphi_s(a_1,\dots ,a_r&;b_1,\dots ,b_s;q,x)\\
&:=\sum_{n\ge 0}\frac{(a_1,\dots ,a_r;q)_n}{(b_1,\dots ,b_s;q)_n(q;q)_n}\left\{(-1)^nq^{\frac{n(n-1)}{2}}\right\}^{1+s-r}x^n.
\end{align*}
The notation $(a;q)_n$ is the $q$-shifted factorial;
\[(a;q)_n:=
\begin{cases}
1, &n=0, \\
(1-a)(1-aq)\dots (1-aq^{n-1}), &n\ge 1,
\end{cases}
\]
moreover, $(a;q)_\infty :=\lim_{n\to \infty}(a;q)_n$ and 
\[(a_1,a_2,\dots ,a_m;q)_\infty:=(a_1;q)_\infty (a_2;q)_\infty \dots (a_m;q)_\infty.\]
The basic hypergeometric series has radius of convergence $\infty , 1$ or $0$ according to whether $r-s<1, r-s=1$ or $r-s>1$. We remark that the solution $u_2(x)$ contains a divergent series and other solutions $u_1(x)$, $v_1(x)$ and $v_2(x)$ are convergent series. We study the relation between these solutions from the viewpoint of connection problems on linear $q$-difference equations. The function $\theta_q(x)$ is the theta function of Jacobi with the base $q$ (see section two for more details).

\medskip
The solution $u_1(x)$ is called the Ramanujan function, which has found by Ramanujan \cite{Ramanujan}.  %% Hereafter, we denote this function as  
%% M.~E.~H.~Ismail \cite{IZ} has pointed out that the Ramanujan function
M.~E.~H.~Ismail \cite{IZ} has introduced the notation 
\[\ram_q(x):=\sum_{n\ge 0}\frac{q^{n^2}}{(q;q)_n}(-x)^n= {}_0\varphi_1(-;0;q,-qx).\]
He shows that $\ram_q(x)$ is one of $q$-analogues of the Airy function \cite{Is}.  The Ramanujan function 
$\ram_q(x)$ appears in the third identity on p.$57$ of Ramanujan's ``Lost notebook'' \cite{Ramanujan} as follows (with $x$ replaced by $q$):
\[\ram_q(-a)=\sum_{n\ge 0}\frac{a^nq^{n^2}}{(q;q)_n}=
\prod_{n\ge 1}\left(1+\frac{aq^{2n-1}}{1-q^ny_1-q^{2n}y_2-q^{3n}y_3-\cdots }\right)\]
where
\begin{align*}
y_1&=\frac{1}{(1-q)\psi^2(q)},\\
y_2&=0,\\
y_3&=\frac{q+q^3}{(1-q)(1-q^2)(1-q^3)\psi^2(q)}
-\frac{\sum_{n\ge 0}\frac{(2n+1)q^{2n+1}}{1-q^{2n+1}}}{(1-q)^3\psi^6(q)},\\
y_4&=y_1y_3,\\
\psi (q)&=\sum_{n\ge 0}q^{\frac{n(n+1)}{2}}=\frac{(q^2;q^2)_\infty}{(q;q^2)_\infty}.\\
\end{align*}
% To be precise, the Ramanujan function is given by
% \[\ram_q(x):=\sum_{n\ge 0}\frac{q^{n^2}}{(q;q)_n}(-x)^n.\]
Strictly speaking, Ramanujan has not shown that the Ramanujan function $\ram_q(x)$ satisfies the equation \eqref{ramanujaneq}. But we propose that the equation \eqref{ramanujaneq} is named ``the Ramanujan equation'' after his study.

We review some $q$-special functions and connection formulae of these functions. Recently, Y.~Ohyama \cite{Ohyama} shows that there exists ``the digeneration diagram'' for Heine's series ${}_2\varphi_1(a,b;c;q,x)$:
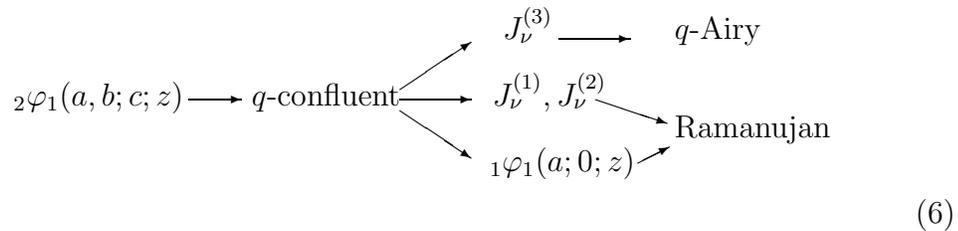
\begin{equation}
	\begin{picture}(370,70)(0,0)
        \put(20,30){{  $ _2 \varphi_1(a,b;c;z)$ } }
        \put(110,30){{ $q$-confluent   }}
        \put(200,5){{ $ _1 \varphi_1(a;0;z)$   }}
        \put(205,55){{ $J_\nu^{(3)}$  }}
        %\put(205,42){ {$J_\nu^{(2)}$}}
        \put(205,30){{$J_\nu^{(1)},J_{\nu}^{(2)} $}}
        \put(270,55){{  $q$-Airy}}
        \put(270,18){{ Ramanujan}}
        \put(90,32){\vector(1,0){20}}
        \put(170,36){\vector(3,2){27}}
        \put(170,32){\vector(1,0){27}}
        \put(230,55){\vector(1,0){27}}
        \put(170,28){\vector(3,-2){27}}
        \put(260,8){\vector(2,1){12}}
        \put(244,32){\vector(3,-1){27}}
     \end{picture}
\label{qdegenerate}
\end{equation}

We remark that there exist three different types of $q$-Bessel functions $J_\nu^{(j)},j=1,2,3$ \cite{GR} and two $q$-analogues of the Airy function. These $q$-analogue are solutions of different types of linear $q$-difference equations. In this point, this diagram is essentially different from the differential case:
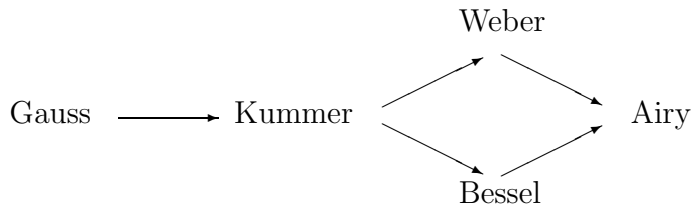
\begin{equation*}
	\begin{picture}(370,70)(0,0)
        \put(35,30){{  Gauss } }
        %\put(20,22){${}_2F_1(a,b;c;z)$ }
        \put(120,30){{  Kummer }}
        %\put(120,22){${}_1F_1(a;c;z)$}
        \put(205,0){{  Bessel}}
        %\put(180,-15){$\frac{(z/2)^\nu}{\Gamma (\nu +1)}{}_0F_1(\nu +1;-z^2/4)$}
        \put(205,65){{  Weber}}
        %\put(215,60){${}_1F_1(-\nu /2;1/2;z^2 )$}
        \put(270,30){{  Airy}}
        %\put(270,22){{(5/2)}}
        \put(80,32){\vector(1,0){38}}
        \put(180,36){\vector(2,1){38}}
        \put(180,30){\vector(2,-1){38}}
        \put(225,10){\vector(2,1){38}}
        \put(225,56){\vector(2,-1){38}}
     \end{picture}\label{digconti}
\end{equation*}

The Ramanujan function is the most degenerated case in the diagram \eqref{qdegenerate}.
We study connection problems on linear $q$-difference equations between around the origin and around the infinity with irregular singular points. The irregularity of $q$-difference equations and $q$-difference modules for some cases of slopes of the Newton polygons are studied by J.-P.~Ramis, J.~Sauloy and C.~Zhang \cite{RSZ}. Graphically, these (formal) invariants are the height of the right part of the Newton polygon, from the bottom to the upper right end. 
In this paper, we show connection formulae for the Ramanujan equation. Connection problems on linear $q$-difference equations between the origin and the infinity are studied by G.~D.~Birkhoff \cite{Birkhoff}. 

Connection formulae of second order linear $q$-difference equations are linear relations in a matrix form:
\[\begin{pmatrix}
u_1(x)\\
u_2(x)
\end{pmatrix}
=
\begin{pmatrix}
C_{11}(x)&C_{12}(x)\\
C_{21}(x)&C_{22}(x)
\end{pmatrix}
\begin{pmatrix}
v_1(x)\\
v_2(x)
\end{pmatrix}.\]
Here, $u_1(x)$ and $u_2(x)$ are solutions around the origin, $v_1(x)$ and $v_2(x)$ are solutions around infinity and $C_{ij}$ $(1\le i,j\le 2)$ are doubly periodic functions such that 
\[\sigma_qC_{ij}(x)=C_{ij}(x),\quad C_{ij}(e^{2\pi i}x)=C_{ij}(x),\]
namely, $q$-elliptic functions. 
The first example of the connection formula with regular singular points was found by G.~N.~Watson \cite{W} in 1912: 
\begin{align}\label{wato}
{}_2 \varphi_1\left(a,b;c;q;x \right)= 
\frac{(b,c/a;q)_\infty (a x,q/ a x;q)_\infty }{(c, b/a;q)_\infty (  x,q/   x;q)_\infty } 
{}_2 \varphi_1\left(a,aq/c;aq/b;q;cq/abx \right) \nonumber \\+ 
\frac{(a,c/b;q)_\infty (b x, q/ b x;q)_\infty }{(c, a/b;q)_\infty (  x,q/   x;q)_\infty } 
{}_2 \varphi_1\left(b, bq/c; bq/a; q; cq/abx \right).\notag
\end{align}
But connection formulae for irregular singular case had not found for a long time. Recently, C.~Zhang gives connection formulae for some confluent type basic hypergeometric series \cite{Z0,Z1,Z2}. In \cite{Z1}, Zhang gives a connection matrix of Jackson's first and second $q$-Bessel function $J_\nu^{(j)}(x;q), (j=1,2)$ \cite{GR} with using the $q$-Borel-Laplace transformations of the second kind. Zhang also gives the connection formula for the divergent seties ${}_2\varphi_0(a,b;-;q,x)$ in \cite{Z0,Z2} with using the $q$-Borel-Laplace transformations of the first kind. These resummation method are powerful tools for connection problems with irregular singular points. We define these transformations later. 

The connection formula for the solution \eqref{secondram}, i.e., $u_2(x)$ has not known. We remark that the basic hypergeometric part ${}_2\varphi_0(0,0;-;q,-x/q)$ in this solution is a divergent series around the origin and the $q$-Stokes phenomenon occurs \cite{sauloy}. Therefore, we need a suitale resummation method, which is the $q$-Borel-Laplace transformations of the first kind. By this resummation method, we obtain the following connection formula in the matrix form exactly and solve the connection problem on the Ramanujan equation.

\bigskip
\noindent
\textbf{Theorem.} \textit{For any $x\in\mathbb{C}^*\setminus [-\lambda ;q]$, we have}
\[\begin{pmatrix}
u_1(x)\\
\tilde{u}_2(x,\lambda )
\end{pmatrix}
=
\begin{pmatrix}
C_{11}(x)&C_{12}(x)\\
\tilde{C}_{21}(x)&\tilde{C}_{22}(x)
\end{pmatrix}
\begin{pmatrix}
v_1(x)\\
v_2(x)
\end{pmatrix},\]
\textit{provided that}
\[C_{11}(x)=\frac{\theta_{q^2}(qx)\theta_{q^2}(x)}{(q,q^2;q^2)_\infty\theta_q(x)}, \quad 
C_{12}(x)=\frac{\theta_{q^2}(x)\theta_{q^2}(x/q)}{(q,q^2;q^2)_\infty\theta_q(x/q)},\]
\[\tilde{C}_{21}(x)=\frac{(q;q)_\infty\theta_{q^2}(-qx/\lambda^2)\theta_{q^2}(x)}{\theta_q(-q/\lambda )\theta_q(x/\lambda )\theta_q(x)}\]
\textit{and}
\[\tilde{C}_{22}(x)=\frac{(q;q)_\infty\theta_{q^2}(-x/\lambda^2)\theta_{q^2}(x/q)}{\theta_q(-1/\lambda )\theta_q(x/\lambda )\theta_q(x/q)}.\]
\textit{Here, the function $\tilde{u}_2(x,\lambda )$ is}
\[\tilde{u}_2(x,\lambda )=\theta_q(x){}_2f_0(0,0;-;q,\lambda ,-x/q)\]
\textit{which is meromorphic function on $\mathbb{C}^*\setminus [-\lambda ;q]$ and the set $[\lambda ;q]$ is the $q$-spiral(see section two). We denote ${}_2f_0(0,0;-;q,\lambda ,-x/q)$ as the resummation of ${}_2\varphi_0(0,0;-;q,-x/q)$.}

\bigskip 
We show this formula with the using of the $q$-Borel-Laplace method of the first kind. The connection formula between the Ramanujan function and the $q$-Airy function has given with the using of the $q$-Borel-Laplace method of the second kind. 

\bigskip
\noindent
\textbf{Theorem }(Morita\cite{M0}) \textit{For any }$x\in\mathbb{C}^*$, we have
\begin{equation}
\ram_{q^2}\left(-\frac{q^3}{x^2}\right)
=
\frac{1}{(q,-1;q)_\infty}\left\{\theta_q\left(\frac{x}{q}\right)\ai_q(-x)
+\theta_q \left(-\frac{x}{q}\right)\ai_q (x)\right\}.\label{ram-qairy}
\end{equation}

Here, the function $\ai_q(x )$ is the $q$-Airy function, which is another $q$-analogue of the Airy function.  The $q$-Airy function 
$\ai_q(x )$ is found  as a special solution   of the second $q$-Painlev\'e equation by  K.~Kajiwara, T.~Masuda, M.~Noumi, Y.~Ohta and Y.~Yamada 
\cite{hama,KMNOY} . The function $\ai_q(x)$ is defined by
%%%o
%%%o "...the $q$-Airy function. This function ..."という表現はくどいので
%%%o 関係代名詞を使う
%%%o
%%%o The $q$-Airy function has found では、$q$-Airy函数が何か見つけた、という
%%%o 意味になる！
%%%o 
%% Here, the function $\ai_q(x )$ is the $q$-Airy function. This function is considered as another $q$-analogue of the Airy function (see section three). The $q$-Airy function has found in the study of the second $q$-Painlev\'e equation\cite{hama, KMNOY} by K.~Kajiwara, T.~Masuda, M.~Noumi, Y.~Ohta and Y.~Yamada. The function $\ai_q(\cdot )$ is defined by
\[\ai_q(x):=\sum_{n\ge 0}\frac{1}{(-q;q)_n(q;q)_n}\left\{(-1)^nq^{\frac{n(n-1)}{2}}\right\}(-x)^n.\]

These functions $\ram_q(x)$ and $\ai_q(x)$ are most degenerated case in the diagram for Heine's ${}_2\varphi_1(a,b;c;q,x)$ series \eqref{qdegenerate} and satisfy different two types of second order $q$-difference equations. The Ramanujan function $\ram_q(x)$ satisfies 
\[qxu(q^2x)-u(qx)+u(x)=0\]
and the $q$-Airy function $\ai_q(x)$ satisfies
\[u(q^2x)+xu(qx)-u(x)=0.\]
%We remark that the linear independent solution of this equation around the orig%in is given by 
%\[u(x)=\frac{\theta_q(q^2x)}{\theta_q(-q^2x)}\ai_q(-x).\]

Ismail also has pointed out the Ramanujan function and the $q$-Airy function are different. But the relation between them has not known. 
Our connection formula \eqref{ram-qairy} shows that the Ramanujan function 
can be represented by the $q$-Airy functions.
%%%o こんな感じですかねー
%% We give the connection formula between these functions in section five.

We remark that the connection formula for $u_1(x)$ is essentially given by Ismail and C.~Zhang as follows \cite{IZ}:
%%%o 
%%%o こういうときはセミコロンを使います
%% We remark that the connection formula for $u_1(x)$ is essentially given by %% Ismail and C.~Zhang as follows\cite{IZ};
\begin{align}
\ram_q(x)&=\frac{(qx,q/x;q^2)_\infty}{(q;q^2)_\infty}{}_1\varphi_1\left(0;q;q^2,\frac{q^2}{x}\right)\notag\\
&-\frac{q(q^2x,1/x;q^2)_\infty}{(1-q)(q;q^2)_\infty}{}_1\varphi_1\left(0;q^3;q^2,\frac{q^3}{x}\right).
\label{Is-Zh}
\end{align}
They give \eqref{Is-Zh} as an asymptotic formula for the Ramanujan function. But from the viewpoint of connection problems on $q$-difference equations, we can regard the formula \eqref{Is-Zh} as one of connection formulae of the Ramanujan function. In fact, we can rewrite this formula as follows:
\begin{align*}
u(x)=\frac{\theta_{q^2}(qx)\theta_{q^2}(x)}{(q,q^2;q^2)_\infty \theta_q(x)}v_1(x)+\frac{\theta_{q^2}(x)\theta_{q^2}(x/q)}{(q,q^2;q^2)_\infty \theta_q(x/q)}v_2(x).
\end{align*}
Here, connection coefficients are $q^2$-elliptic functions. 
In \cite{M0}, we derive Ismail-Zhang's formula \eqref{Is-Zh} from \eqref{ram-qairy} by suitable algebraic transformation.

In the last section, we introduce the $q$-Borel-Laplace transformations of level $r-1$. These transformations are higher order extension of the $q$-Borel-Laplace transformations. We also apply these new method to a divergent series ${}_r\varphi_0(0,0,\dots ,0;-;q,x)$ and give the asymptotic formula as follows:

\bigskip
\noindent
\textbf{Theorem.} For any $\mathbb{C}^*\setminus [-\lambda ;q^{r-1}]$, we have
\begin{align}&\left(\mathcal{L}_{q^{r-1,\lambda}}^+\circ \mathcal{B}_{q^{r-1}}^+{}_r\varphi_0(0,0,\dots ,0;-;q,x)\right)\notag\\
&=\frac{1}{\theta_{q^{r-1}}\left(\frac{\lambda}{x}\right)}\frac{(q;q)_\infty}{\theta_q(-\lambda )}\theta_{q^{r(r-1)}\left((-1)^{r-1}q^{\frac{(r-1)(r-2)}{2}}\left(\frac{\lambda}{x}\right)^r\right)}\notag\\
&\times{}_{r-1}\varphi_{r-1}\left(0,0,\dots ,0;q,q^2,\dots ,q^{r-1};q^r,q^{\frac{(r-1)(r-2)}{2}}/x\right)\notag\\
&+\dots \notag\\
&+\frac{1}{\theta_{q^{r-1}}\left(\frac{\lambda}{x}\right)}\frac{(q;q)_\infty}{\theta_q(-\lambda )}\theta_{q^{r(r-1)}\left((-1)^{r-1}q^{\frac{3r(r-1)}{2}}\left(\frac{\lambda}{x}\right)^r\right)}\notag\\
&\times \frac{(-1)^{r-1}q^{\frac{(r-1)(r-2)}{2}}}{(1-q)(1-q^2)\dots (1-q^{r-1})}\left(\frac{q}{\lambda}\right)^{r-1}\notag\\
&\times{}_{r-1}\varphi_{r-1}\left(0,0,\dots ,0;q^r,q^{r+1},\dots ,q^{2r-1};q^r,q^{\frac{3r(r-1)}{2}}/x\right)\notag
\end{align}
%%%%%%%%%%%%%%%%%%%%%%%%%%%%%%section 2%%%%%%%%%%%%%%%%%%%%%%%%%%%%%%%%%%%%%%%

\section{Basic notations}
In this section, we fix our notations. We assume that $q\in\mathbb{C}^*$ satisfies $0<|q|<1$. The $q$-shifted operator $\sigma_q$ is given by $\sigma_qf(x)=f(qx)$. For any fixed $\lambda\in\mathbb{C}^*\setminus q^{\mathbb{Z}}$, the set $[\lambda ;q]$-spiral is $[\lambda ;q]:=\lambda q^{\mathbb{Z}}=\{\lambda q^k;k\in\mathbb{Z}\}$. 
The theta function of Jacobi is important in connection problems. The theta function of Jacobi with the base $q$ is
\[\theta_q(x):=\sum_{n\in\mathbb{Z}}q^{\frac{n(n-1)}{2}}x^n,\qquad \forall x\in\mathbb{C}^*.\]
The theta function has the following properties;
\begin{enumerate}
\item Jacobi's triple product identity is
\[\theta_q(x)=\left(q,-x,-\frac{q}{x};q\right)_\infty .\]
\item The $q$-difference equation which the theta function satisfies;
\[\theta_q(q^kx)=q^{-\frac{n(n-1)}{2}}x^{-k}\theta_q(x),\quad \forall k\in\mathbb{Z}.\]
\item The inversion formula;
\[\theta_q\left(\frac{1}{x}\right)=\frac{1}{x}\theta_q(x).\]
\end{enumerate}
We remark that the function $\theta (-\lambda x)/\theta (\lambda x)$, $\forall\lambda\in\mathbb{C}^*$ satisfies a $q$-difference equation
\[u(qx)=-u(x)\]
which is also satisfied by the function $u(x)=e^{\pi i\left(\frac{\log x}{\log q}\right)}$.

\section{Two $q$-exponential functions}
We review two different $q$-exponential functions $e_q(x)$ and $E_q(x)$ to consider the connection problem on the Ramanujan equation. In this section, we show the relation between them. We review two different $q$-exponential functions from the viewpoint of the connection problems. One of the $q$-exponential function $e_q(x)$ is given by 
\[e_q(x):={}_1\varphi_0 (0;-;q,x)=\sum_{n\ge 0}\frac{x^n}{(q;q)_n}.\]
The other $q$-exponential function $E_q(x)$ is 
\[E_q(x):={}_0\varphi_0(-;-;q,-x)=\sum_{n\ge 0}\frac{q^{\frac{n(n-1)}{2}}}{(q;q)_n}x^n.\]
The function $e_q(x)$ satisfies the following first order $q$-difference equation
\[\left\{\sigma_q-(1-x)\right\}u(x)=0\]
and $E_q(x)$ satisfies 
\[\left\{(1+x)\sigma_q-1\right\}u(x)=0.\]
The limit $q\to 1-0$ converges the exponential function
\[\lim_{q\to 1-0}e_q\left(x(1-q)\right)=\lim_{q\to 1-0}E_q\left(x(1-q)\right)=e^x.\]
In this sense, these functions considered as $q$-analogues of the exponential function. It is known that there exists the relation between these functions \cite{GR}:
\[e_q(x)E_q(-x)=1,\quad e_{q^{-1}}(x)=E_q(-qx).\]
But another relation has not known. We show the connection formula between them and give alternate representation of $e_q(x)$.

At first, we show the following connection formula between $e_q(x)$ and $E_q(x)$. 
\begin{thm}\label{eqEq} For any $x\in\mathbb{C}^*\setminus [1;q]$, 
\[e_q(x)=\frac{(q;q)_\infty}{\theta_q(-x)}E_q\left(-\frac{q}{x}\right)\]
where $|x|<1$.
\end{thm}
\begin{proof}The function $e_q(x)$ and $E_q(x)$ have infinite product as follows:
\[e_q(x)=\frac{1}{(x;q)_\infty},\qquad |x|<1\]
and
\[E_q(x)=(-x;q)_\infty .\]
We remark that $e_q(x)$ can be described as
\[e_q(x)=\frac{1}{\theta_q(-x)}\left(q,\frac{q}{x};q\right)_\infty =\frac{(q;q)_\infty}{\theta_q(-x)}E_q\left(-\frac{q}{x}\right)\]
where $|x|<1$. We obtain the conclusion.
\end{proof}
Therefore, these $q$-exponential functions are related by the connection formula between the origin and the infinity. If we replace $x$ by $x/q$, we obtain the following lemma. This is useful to consider the connection problem in the last section. 
\begin{lmm}\label{alt}For any $x\in\mathbb{C}^*\setminus [1;q]$, the function $e_q(x/q)$ has the following alternate representation.
\begin{align*}
e_q\left(\frac{x}{q}\right)=
%{}_1\varphi_0\left(0;-;q,\frac{x}{q}\right)&=
\frac{(q;q)_\infty}{\theta_q\left(-\frac{x}{q}\right)}{}_0\varphi_1\left(-;q;q^2,\frac{q^5}{x^2}\right)
-\frac{(q;q)_\infty}{\theta_q\left(-\frac{x}{q}\right)}\frac{q^2}{(1-q)x}{}_0\varphi_1\left(-;q^3;q^2,\frac{q^7}{x^2}\right).
\end{align*}
\end{lmm}
\begin{proof}From theorem \ref{eqEq}, 
\[{}_1\varphi_0\left(0;-;q,\frac{x}{q}\right)=\frac{(q;q)_\infty}{\theta_q\left(-\frac{x}{q}\right)}E_q\left(-\frac{q^2}{x}\right)=\frac{(q;q)_\infty}{\theta_q\left(-\frac{x}{q}\right)}{}_0\varphi_0\left(-;-;q,\frac{q^2}{x}\right).\]
Here,
\[{}_0\varphi_0\left(-;-;q,\frac{q^2}{x}\right)
=\sum_{k\ge 0}\frac{1}{(q;q)_k}(-1)^kq^{\frac{k(k-1)}{2}}\left(\frac{q^2}{x}\right)^k\]
and we remark that $(a;q)_{2k}=(a,aq;q^2)_k$ \cite{GR}. By separating the terms with even and odd $k\ge 0$, we obtain the conclusion.
\end{proof}
By separating the terms with $r$-th order terms, we have the following corollary.
\begin{crl}\label{cor33}For any $x\in\mathbb{C}^*\setminus [1;q]$, we have
\begin{align}e_q(x)&=\frac{(q;q)_\infty}{\theta_q(-x)}
\left[{}_0\varphi_{r-1}\left(-;q,q^2,\dots ,q^{r-1};q^r,(q^r)^{\frac{r-1}{2}}\left(\frac{q}{x}\right)^r\right)\right.\notag\\
&+\frac{(-1)}{1-q}\frac{q}{x}{}_0\varphi_{r-1}\left(-;q^2,\dots ,q^{r-1},q^{r+1};q^r,(q^r)^{\frac{r+1}{2}}\left(\frac{q}{x}\right)^r\right)\notag\\
&+\dots +\frac{(-1)^{r-1}q^{\frac{(r-1)(r-2)}{2}}}{(1-q)(1-q^2)\dots (1-q^{r-1})}\left(\frac{q}{x}\right)^{r-1}\notag\\
&\left.{}_0\varphi_{r-1}\left(-;q^{r+1},\dots ,q^{2(r-1)};q^r,(q^r)^{\frac{3(r-1)}{2}}\left(\frac{q}{x}\right)^r\right)
\right].\notag
\end{align}
\end{crl}

%%%%%%%%%%%%%%%%%%%%%%section 3%%%%%%%%%%%%%%%%%%%%%%%%%%%%%%%%%%%%%%%%%

\section{Covering transformations}
We define a covering transformation for a second order linear $q$-difference equation.
\begin{dfn}For a $q$-difference equation
\begin{equation}\label{sh}
a(x)u(q^2x)+b(x)u(qx)+c(x)u(x)=0,
\end{equation}
we define the covering transformation as follows
\[t^2:=x,\quad v(t):=u(t^2),\quad p:=\sqrt{q}.\]
\end{dfn}
The covering transform of the equation \eqref{sh} is given by
\[a(t^2)v(p^2t)+b(t^2)v(pt)+c(t^2)v(t)=0.\]
By the covering transformation, the equation 
\[\left(K\cdot x\sigma_q^2-\sigma_q+1\right)u(x)=0\]
is transformed to
\begin{equation}
\left(K\cdot t^2\sigma_p^2-\sigma_p+1\right)v(t)=0,\label{kram}
\end{equation}
where $K$ is a fixed constant in $\mathbb{C}^*$.

\section{The Ramanujan function and the $q$-Airy function}

There are two different $q$-analogue of the Airy function. One is called the Ramanujan function which appears in \cite{Ramanujan}. Ismail \cite{Is} pointed out that the Ramanujan function can be considered as a $q$-analogue of the Airy function. The other one is called the $q$-Airy function which is obtained by K.~Kajiwara, T.~Masuda, M.~Noumi, Y.~Ohta and Y.~Yamada \cite{KMNOY}. In this section, we see the properties of these functions. We explain the reason why they are called $q$-analogues of the Airy function and we show $q$-difference equations which they satisfy.The Ramanujan function appears in Ramanujan's ``Lost notebook''~\cite{Ramanujan}.
Ismail has pointed out that the Ramanujan function can be considered as a $q$-analogue of the Airy function.
The Ramanujan function is defined by following convergent series;
\[\ram_q(x):=\sum_{n\ge 0}\frac{q^{n^2}}{(q;q)_n}(-x)^n
={_0\varphi_1}(-;0;q,-qx).\]

In the theory of ordinary differencial equations, the term Plancherel-Rotach asymptotics refers to asymptotics around the largest and smallest zeros. With $x=\sqrt{2n+1}-2^{\frac{1}{2}}3^{\frac{1}{3}}n^{\frac{1}{6}}t$ and for $t\in\mathbb{C}$, the Plancherel-Rotach asymptotic formula for Hermite polynomials $H_n(x)$ is
\begin{equation}
\lim_{n\to +\infty}\frac{e^{-\frac{x^2}{2}}}{3^{\frac{1}{3}}\pi^{-\frac{3}{4}}2^{\frac{n}{2}+\frac{1}{4}}\sqrt{n!}}H_n(x)=\ai (t). \label{pr}
\end{equation}
In \cite{Is}, Ismail shows the $q$-analogue of \eqref{pr};
\begin{prp}We have  
\[\lim_{n\to\infty}\frac{q^{n^2}}{t^n}h_n(\sinh\xi_n|q)=\ram_q\left(\frac{1}{t^2}\right)\]
where $e^{\xi_n}=tq^{-\frac{n}{2}}$.
\end{prp}
Here, $h_n(\cdot |q)$ is the $q$-Hermite polynomial. 
In this sense, we can deal with the Ramanujan function $\ram_q(x)$ as a $q$-analogue
of the Airy function. The Ramanujan function satisfies the following $q$-diference equation;
\begin{equation}
\left(qx\sigma_q^2-\sigma_q+1\right)u(x)=0.
\label{ram}
\end{equation}

\begin{rem}We remark that another solution of the equation \eqref{ram} is given by
\[u(x)=\theta (x){}_2\varphi_0(0,0;-;q,-x/q).\]
Here, 
\[{}_2\varphi_0\left(0,0;-;q,-\frac{x}{q}\right)=\sum_{n\ge 0}\frac{1}{(q;q)_n}\left\{(-1)^nq^{\frac{n(n-1)}{2}}\right\}^{-1}\left(-\frac{x}{q}\right)^n\]
is a divergent series.
\end{rem}

The $q$-Airy function is found by K.~Kajiwara, T.~Masuda, M.~Noumi, Y.~Ohta and Y.~Yamada \cite{KMNOY}, in their study of the $q$-Painlev\'e equations. This function is the special solution of the second $q$-Painlev\'e equations and given by the following series
\[\ai_q(x):=\sum_{n\ge 0}\frac{1}{(-q,q;q)_n}\left\{(-1)^nq^\frac{n(n-1)}{2}\right\}(-x)^n={}_1\varphi_1(0;-q;q,-x).\]

T. Hamamoto, K. Kajiwara, N. S. Witte \cite{hama} proved following asymptotic expansions;

\begin{prp}
With $q=e^{-\frac{\delta^3}{2}}$, $x=-2ie^{-\frac{s}{2}\delta^2}$ as $\delta\to 0$,
\[{_1\varphi_1}(0;-q;q,-qx)=2\pi^{\frac{1}{2}}\delta^{-\frac{1}{2}}
e^{-\left(\frac{\pi i}{\delta^3}\right)\ln 2+\left(\frac{\pi i}{2\delta }\right)s+\frac{\pi i}{12}}\left[\ai\left(se^{\frac{\pi i}{3}}\right)+O(\delta^2)\right],\]
\[{_1\varphi_1}(0;-q;q,qx)=2\pi^{\frac{1}{2}}\delta^{-\frac{1}{2}}
e^{-\left(\frac{\pi i}{\delta^3}\right)\ln 2-\left(\frac{\pi i}{2\delta }\right)s-\frac{\pi i}{12}}\left[\ai\left(se^{-\frac{\pi i}{3}}\right)+O(\delta^2)\right]\]
for $s$ in any compact domain of $\mathbb{C}$.  
\end{prp}

Here, $\ai (x)$ is the Airy function. From this proposition, we can regard the $q$-Airy function as a $q$-analogue of the Airy function.

We can easily check out that the $q$-Airy function satisfies the second order linear $q$-difference equation
\begin{equation}
\left(\sigma_q^2+x\sigma_q-1\right)u(x)=0.
\label{qai}
\end{equation}
Another solution of the equation \eqref{qai} is given by
\[u(x)=
e^{\pi i\left(\frac{\log x}{\log q}\right)}{_1\varphi_1}(0;-q;q,x)
=e^{\pi i\left(\frac{\log x}{\log q}\right)}\ai_q(-x).\]

%%%%%%%%%%%%%%%%%%%%%%%%%%%%%section 4%%%%%%%%%%%%%%%%%%%%%%%%%%%%%%%%%%%%%%%

\section{The $q$-Borel-Laplace transformations}
In this section, we show a connection formula of the divergent series $ {}_2\varphi_0$. This series appears in the second solution of the Ramanujan equation \eqref{ram}. At first, we define two types of the $q$-Borel-Laplace transformations.

\medskip
\noindent
\textbf{Definition.} \textit{We assume that $f(x)$ is a formal power series $f(x)=\sum_{n\ge 0}a_nx^n$, $a_0=1$.}
\begin{enumerate}
\item The $q$-Borel transformation of the first kind is
\[\left(\mathcal{B}_q^+f\right)(\xi ):=\sum_{n\ge 0}a_nq^{\frac{n(n-1)}{2}}\xi^n\left(=:\varphi (\xi )\right).\]
\item The $q$-Laplace transformation of the first kind is
\[\left(\mathcal{L}_{q, \lambda}^+\varphi\right)(x):=
\frac{1}{1-q}\int_0^{\lambda\infty}\frac{\varphi (\xi )}{\theta_q\left(\frac{\xi}{x}\right)}\frac{d_q\xi}{\xi}=\sum_{n\in\mathbb{Z}}\frac{\varphi (\lambda q^n)}{\theta_q\left(\frac{\lambda q^n}{x}\right)},\]
here, this transformation is given by Jackson's $q$-integral \cite{GR}.
\end{enumerate}

We also define the $q$-Borel-Laplace transformations of the second kind as follows:
\begin{enumerate}
\item The $q$-Borel transformation of the second kind is 
\[(\mathcal{B}_q^-f)(\xi ):=\sum_{n\ge 0}a_nq^{-\frac{n(n-1)}{2}}\xi^n\left(=:g(\xi )\right).\]
\item The $q$-Laplace transformation of the second kind is
\[\left(\mathcal{L}_q^-g\right)(x):=\frac{1}{2\pi i}\int_{|\xi |=r}g(\xi )\theta_q\left(\frac{x}{\xi }\right)\frac{d\xi}{\xi},\]
where $r>0$ is enough small number. 
\end{enumerate}

We remark that the $q$-Borel transformation $\mathcal{B}_q^+$ is formal inverse of the $q$-Laplace transformation $\mathcal{L}_{q,\lambda }^+$ as follows;
\begin{lmm}For any entire function $f(x)$, we have 
\[\mathcal{L}_{q, \lambda}^+\circ\mathcal{B}_q^+f=f.\]
\end{lmm}
The $q$-Borel transformation $\mathcal{B}_q^-$ also can be considered as a formal inverse of the $q$-Laplace transformation.
\begin{lmm} We assume that the function $f$ can be $q$-Borel transformed to the analytic function $g(\xi )$ around $\xi =0$. Then, we have
\[\mathcal{L}^-_q\circ\mathcal{B}^-_qf=f.\]
\end{lmm}
\begin{proof}We can prove this lemma calculating residues of the $q$-Laplace transformation around the origin. 
\end{proof}

The $q$-Borel transformation $\mathcal{B}_q^-$ has following operational relation.
\begin{lmm}\label{orqb}
For any $l,m\in\mathbb{Z}_{\ge 0}$,
\[\mathcal{B}_q^-(t^m\sigma_q^l)=q^{-\frac{m(m-1)}{2}}\tau^m\sigma_q^{l-m}\mathcal{B}_q^-.\]
\end{lmm}
In the following subsection, we apply these resummation methods to deal with the connection problem.

\section{The connection formula of the series ${}_2\varphi_0(0,0;-;q,\cdot )$}
%In this section, we give the connection formula of the divergent series ${}_2\v%arphi_0(0,0;-;q,\cdot )$. 
The aim of this section is to give a proof for the following theorem;
%We show the following connection formula. 
\begin{thm}\label{divergentramanujan}For any $x\in\mathbb{C}^*\setminus [-\lambda ;q],$ 
\begin{align*}\theta_q(x){}_2f_0&\left(0,0;-;q,-\frac{x}{q}\right)=(q;q)_\infty\frac{\theta_q(x) \theta_{q^2}\left(-\frac{\lambda^2}{qx}\right)}{\theta_q\left(-\frac{\lambda}{q}\right)\theta_q\left(\frac{\lambda}{x}\right)}
{}_1\varphi_1\left(0;q;q^2,\frac{q^2}{x}\right)\\
&+\frac{(q;q)_\infty}{1-q}\frac{\theta_q(x) \theta_{q^2}\left(-\frac{\lambda^2}{x}\right)}{\theta_q\left(-\frac{\lambda}{q}\right)\theta_q\left(\frac{\lambda}{x}\right)}\frac{\lambda}{x}
{}_1\varphi_1\left(0;q^3;q^2,\frac{q^3}{x}\right).
\end{align*} 
\end{thm}

We give the proof of theorem \ref{divergentramanujan}.
\begin{proof}
We apply the $q$-Borel transformation $\mathcal{B}_q^+$ to the divergent 
series $v(x)={}_2\varphi_0(0,0;-;q,-x/q)$. We obtain 
\[\left(\mathcal{B}_q^+v\right)(\xi )={}_1\varphi_0\left(0;-;q,\frac{\xi}{q}\right)=:\varphi (\xi ).\]
From lemma \ref{alt}, 
\[\varphi (\xi )=\frac{(q;q)_\infty}{\theta_q\left(-\frac{\xi}{q}\right)}
{}_0\varphi_1\left(-;q;q^2,\frac{q^5}{\xi^2}\right)-\frac{(q;q)_\infty}{\theta_q\left(-\frac{\xi}{q}\right)}
\frac{q^2}{(1-q)\xi}{}_0\varphi_1\left(-;q^3;q^2,\frac{q^7}{\xi^2}\right)\]
where $|\xi /q|<1$. 

We apply the $q$-Laplace transformation $\mathcal{L}_{q, \lambda}^+$ to $\varphi (\xi )$:
\begin{align*}
&\left(\mathcal{L}_{q,  \lambda}^+\varphi\right)(x)=\sum_{n\in\mathbb{Z}}\frac{\varphi (\lambda q^n)}{\theta_q\left(\frac{\lambda q^n}{x}\right)}
=\sum_{n\in\mathbb{Z}}\frac{{}_1\varphi_0\left(0;-;q,\frac{\lambda q^n}{q}\right)}{\theta_q\left(\frac{\lambda q^n}{x}\right)}\\
%&=(q;q)_\infty\sum_{n\in\mathbb{Z}}\frac{{}_0\varphi_1\left(-;q;q^2,\frac{q^5}{%(\lambda q^n)^2}\right)}{\theta_q\left(-\frac{\lambda q^n}{q}\right)\theta_q\le%ft(\frac{\lambda q^n}{x}\right)}
%-\frac{(q;q)_\infty q^2}{(1-q)\lambda}\sum_{n\in\mathbb{Z}}\frac{{}_0\varphi_1\%left(-;q^3;q^2,\frac{q^7}{(\lambda q^n)^2}\right)}{\theta_q\left(-\frac{\lambda% q^n}{q}\right)\theta_q\left(\frac{\lambda q^n}{x}\right)}\cdot \frac{1}{q^n}\\%
%&=(q;q)_\infty\sum_{n\in\mathbb{Z}}\frac{1}{\theta_q\left(-\frac{\lambda q^n}{q%}\right)\theta_q\left(\frac{\lambda q^n}{x}\right)}
%\sum_{m\ge 0}\frac{(q^2)^{m(m-1)}}{(q,q^2;q^2)_m}\left(\frac{q^{5-2n}}{\lambda^%2}\right)^m\\
%&-\frac{(q;q)_\infty q^2}{(1-q)\lambda}\sum_{n\in\mathbb{Z}}\frac{q^{-n}}{\thet%a_q\left(-\frac{\lambda q^n}{q}\right)\theta_q\left(\frac{\lambda q^n}{x}\right%)}
%\sum_{m\ge 0}\frac{(q^2)^{m(m-1)}}{(q^3,q^2;q^2)_m}\left(\frac{q^{7-2n}}{\lambd%a^2}\right)^m\\
%&=(q;q)_\infty\sum_{n\in\mathbb{Z}}\frac{\left(-\frac{\lambda}{q}\right)^n\left%(\frac{\lambda}{x}\right)^n q^{n(n-1)-2mn}}{\theta_q\left(-\frac{\lambda}{q}\ri%ght)\theta_q\left(\frac{\lambda}{x}\right)}\sum_{m\ge 0}\frac{(q^2)^{m(m-1)}}{(%q,q^2;q^2)_m}\left(\frac{q^5}{\lambda^2}\right)^m\\
%&-\frac{(q;q)_\infty q^2}{(1-q)\lambda}\sum_{n\in\mathbb{Z}}
%\frac{\left(-\frac{\lambda}{q}\right)^n\left(\frac{\lambda}{x}\right)^nq^{-n}\c%dot q^{n(n-1)-2mn}}{\theta_q\left(-\frac{\lambda}{q}\right)\theta_q\left(\frac{%\lambda}{x}\right)}\sum_{m\ge 0}\frac{(q^2)^{m(m-1)}}{(q^3,q^2;q^2)_m}\left(\fr%ac{q^7}{\lambda^2}\right)^m\\
&=\frac{(q;q)_\infty}{\theta_q\left(-\frac{\lambda}{q}\right)\theta_q\left(\frac{\lambda}{x}\right)}\sum_{n-m\in\mathbb{Z}}(q^2)^{\frac{(n-m)(n-m-1)}{2}}\left(-\frac{\lambda^2}{qx}\right)^{n-m}\\
&\qquad\qquad\qquad\qquad \times \sum_{m\ge 0}\frac{(-1)^m(q^2)^{\frac{m(m-1)}{2}}}{(q;q^2;q^2)_m}\left(\frac{q^2}{x}\right)^m\\
&-\frac{(q;q)_\infty}{\theta_q\left(-\frac{\lambda}{q}\right)\theta_q\left(\frac{\lambda}{x}\right)}\frac{q^2}{(1-q)\lambda}\sum_{n-m\in\mathbb{Z}}(q^2)^{\frac{(n-m)(n-m-1)}{2}}\left(-\frac{\lambda^2}{q^2x}\right)^{n-m}\\
&\qquad\qquad\qquad\qquad \times
\sum_{m\ge 0}\frac{(-1)^m(q^2)^{\frac{m(m-1)}{2}}}{(q^3,q^2;q^2)_m}\left(\frac{q^3}{x}\right)^m.
\end{align*}
Therefore, 
\begin{align*}
&{}_2f_0\left(0,0;-;q,-\frac{x}{q}\right)
=\mathcal{L}_{q, \lambda}^+\circ\mathcal{B}_q^+{}_2\varphi_0\left(0,0;-;q,-\frac{x}{q}\right)\\
&=(q;q)_\infty\frac{\theta_{q^2}\left(-\frac{\lambda^2}{qx}\right)}{\theta_q\left(-\frac{\lambda}{q}\right)\theta_q\left(\frac{\lambda}{x}\right)}{}_1\varphi_1\left(0;q;q^2,\frac{q^2}{x}\right)
+\frac{(q;q)_\infty}{1-q}\frac{\theta_{q^2}\left(-\frac{\lambda^2}{x}\right)}{\theta_q\left(-\frac{\lambda}{q}\right)\theta_q\left(\frac{\lambda}{x}\right)}
{}_1\varphi_1\left(0;q^3;q^2,\frac{q^3}{x}\right).
\end{align*}
We obtain the conclusion.
\end{proof}

\begin{rem}By theorem \ref{divergentramanujan}, we have
\[\tilde{u}_2(x,\lambda )=\tilde{C}_{21}v_1(x)+\tilde{C}_{22}v_2(x),\]
where
\[\tilde{C}_{21}(x)=\frac{(q;q)_\infty\theta_{q^2}(-qx/\lambda^2)\theta_{q^2}(x)}{\theta_q(-q/\lambda )\theta_q(x/\lambda )\theta_q(x)}\]
and
\[\tilde{C}_{22}(x)=\frac{(q;q)_\infty\theta_{q^2}(-x/\lambda^2)\theta_{q^2}(x/q)}{\theta_q(-1/\lambda )\theta_q(x/\lambda )\theta_q(x/q)}.\]
This is a half of our connection formula.
\end{rem}
%We put
%\begin{align}
%C_{11}(x)&=\frac{\theta_{q^2}(qx)\theta_{q^2}(x)}{(q,q^2;q^2)_\infty\theta_q(x)%}, \notag \\ 
%C_{12}(x)&=\frac{\theta_{q^2}(x)\theta_{q^2}(x/q)}{(q,q^2;q^2)_\infty\theta_q(x%/q)},\notag \\
%\tilde{C}_{21}(x)&=\frac{(q;q)_\infty\theta_{q^2}(-qx/\lambda^2)\theta_{q^2}(x)%}{\theta_q(-q/\lambda )\theta_q(x/\lambda )\theta_q(x)} \notag 
%\end{align}
%and
%\[\tilde{C}_{22}(x)=\frac{(q;q)_\infty\theta_{q^2}(-x/\lambda^2)\theta_{q^2}(x/%q)}{\theta_q(-1/\lambda )\theta_q(x/\lambda )\theta_q(x/q)}.\]
%Combining Ismail-Zhang's theorem \cite{IZ}, we obtain the connection matrix for% the Ramanujan equation. 

%%%%%%%%%%%%%%%%%%%%%%%%%%%%%%%%%%section4.4%%%%%%%%%%%%%%%%%%%%%%%%%%%%%%%%%%%%
\section{The $q$-Airy equation around the infinity and the connection formula}
We consider the behavior of the equation \eqref{qai} around the infinity. We set $x=1/t$ and $z(t)=u(1/t)$. Then $z(t)$ satisfies 
\[\left(-\sigma_q^2+\frac{1}{q^2t}\sigma_q+1\right)z(t)=0.\]
We set $\mathcal{E}(t)=1/\theta (-q^2t)$ and $f(t)=\sum_{n\ge 0}a_nt^n,\quad a_0=1$. We assume that $z(t)$ can be described as 
\[z(t)=\mathcal{E}(t)f(t)=\frac{1}{\theta (-q^2t)}\left(\sum_{n\ge 0}a_nt^n\right).\]

The function $\mathcal{E}(t)$ has the following property;
\begin{lmm} For any $t\in\mathbb{C}^*$,
\[\sigma_q\mathcal{E}(t)=-q^2t\mathcal{E}(t),
\qquad \sigma_q^2\mathcal{E}(t)=q^5t^2\mathcal{E}(t).\]
\end{lmm}
From this lemma, $f(t)$ satisfies the following equation
\begin{equation}\label{qaeq}
\left(-q^5t^2\sigma_q^2-\sigma_q+1\right)f(t)=0.
\end{equation}
Since \eqref{qaeq} is the same as \eqref{kram} for $K=-q^5$, we obtain
\[f(t)={}_0\varphi_1(-;0;q^2,q^5t^2)=\ram_{q^2}(-q^3t^2).\]

%\subsection{The $q$-Borel transformation, the $q$-Laplace transformation and th%e connection formula}
%In this section, 
We show a connection formula for $f(t)$. In order to obtain a connection formula, we need the $q$-Borel transformation and the $q$-Laplace transformation following Zhang \cite{Z1}.
%\subsection{The connection formula of the $q$-Airy function}
Applying the $q$-Borel transformation to the equation \eqref{kram} and using lemma \ref{orqb}, we obtain the first order $q$-difference equation
\[g(q\tau )=(1+q^2\tau )(1-q^2\tau )g(\tau ).\]
Since $g(0)=1$, $g(\tau )$ is given by an infinite product
\[g(\tau )=\frac{1}{(-q^2\tau ;q)_\infty(q^2\tau ;q)_\infty}\]
which has single poles at
\[\left\{\tau ;\tau =\pm q^{-2-k},\quad \forall k\in\mathbb{Z}_{\ge 0}\right\}.\]
By Cauchy's residue theorem, the $q$-Laplace transform of $g(\tau )$ is 
\begin{align*}
f(t)=&\frac{1}{2\pi i}\int_{|\tau |=r}g(\tau )\theta\left(\frac{t}{\tau}\right)\frac{d\tau}{\tau }\\
=&-\sum_{k\ge 0}\res\left\{g(\tau )\theta\left(\frac{t}{\tau}\right)\frac{1}{\tau };\tau =-q^{-2-k}\right\}\\
&-\sum_{k\ge 0}\res\left\{g(\tau )\theta\left(\frac{t}{\tau}\right)\frac{1}{\tau };\tau =q^{-2-k}\right\}
\end{align*}
where $0<r<r_0:=1/|q^2|$. We can culculate the residue from lemma \ref{lems}.
% and lemma \ref{qt} .
\begin{lmm}\label{lems}For any $k\in\mathbb{N}$, $\lambda\in\mathbb{C}^*$, we have:
\begin{enumerate}
%\item $\theta_q(q^nt)=q^{-\frac{n(n-1)}{2}}t^{-n}\theta_q(t)$,
\item $\res\left\{\dfrac{1}{\left(\tau /\lambda ;q\right)_\infty}\dfrac{1}{\tau}:\tau =\lambda q^{-k}\right\}
=\dfrac{(-1)^{k+1}q^{\frac{k(k+1)}{2}}}{(q;q)_k (q;q)_\infty}$,
\item $\dfrac{1}{(\lambda q^{-k};q)_\infty}=\dfrac{(-\lambda )^{-k}q^{\frac{k(k+1)}{2}}}{(\lambda ;q)_\infty \left(q/\lambda ;q\right)_k},\quad \lambda \not \in q^{\mathbb{Z}}$.
\end{enumerate}
\end{lmm}
Summing up all of residues, we obtain 
\begin{align*}
f(t)&=\frac{\theta (q^2t)}{(q,-1;q)_\infty} {_1\varphi_1}\left(0,-q;q,\frac{1}{t}\right)
+\frac{\theta (-q^2 t)}{(q,-1;q)_\infty} {_1\varphi_1}\left(0,-q;q,-\frac{1}{t}\right).
\end{align*}
%Combining with lemma\ref{inv}, 
We obtain a connection formula for $z(t)=\mathcal{E}(t)f(t)$. Finally, we acquire the following connection formula between the Ramanujan function and the $q$-Airy function.
\begin{thm}
For any $x\in\mathbb{C}^*$,
\[\ram_{q^2}\left(-\frac{q^3}{x^2}\right)
=\frac{1}{(q,-1;q)_\infty}
\left\{\theta \left(\frac{x}{q}\right)\ai_q(-x)+\theta\left(-\frac{x}{q}\right)\ai_q(x)\right\}.\]
\end{thm}
Here, both $\ram_q(x)$ and $\ai_q(x)$ are defined by convergent series on whole of the complex plain. The connection formula above is valid for any $x\in\mathbb{C}^*$.

%%%%%%%%%%%%%%%%%%%%%%section5.4%%%%%%%%%%%%%%%%%%%%%%%%%%%%%%%%%%%%%%%%
\section{The $q$-Borel-Laplace transformations of level $r-1$ and its application}
In this section, we define the higher order extension of the $q$-Borel-Laplace method. We also apply this method to the divergent series ${}_r\varphi_0(0,0,\dots ,0;-;q,x)$. At first, we give the definition of the $q$-Borel-Laplace transformations of level $r-1$ as follows:

\begin{dfn}The $q$-Borel transformation of level $r-1$ is
\[\left(\mathcal{B}_{q^{r-1}}^+f\right)(\xi ):=\sum_{n\ge 0}a_n(q^{r-1})^{\frac{n(n-1)}{2}}\xi^n=:\hat{\varphi}(\xi ),\]
where $f(x)$ is a formal power series. The $q$-Laplace transformation of level $r-1$ is
\[\left(\mathcal{L}_{q^{r-1}, \lambda}^+\hat{\varphi}\right)(x):=
\sum_{n\in\mathbb{Z}}\frac{\hat{\varphi} (\lambda q^{(r-1)n})}{\theta_{q^{r-1}}\left(\frac{\lambda q^{(r-1)n}}{x}\right)}.\]
\end{dfn}
We also have the following lemma.
\begin{lmm}For any entire function $f(x)$, we have
\[\mathcal{L}_{q^{r-1},\lambda}^+\circ\mathcal{B}_{q^{r-1}}^+f=f.\]
\end{lmm}

We study the application of this method. We consider the following divergent series
\[{}_r\varphi_0(0,0,\dots ,0;-;q,x)=\sum_{n\ge 0}\frac{1}{(q;q)_n}\left\{(-1)^nq^{\frac{n(n-1)}{2}}\right\}^{1-r}x^n.\]
We apply the $q$-Borel transformation of level $r-1$ to this series. Then, we obtain the $q$-exponential function $e_q(x)$. By corollariy \ref{cor33}, the $q$-Borel transform of ${}_r\varphi_0(0,0,\dots ,0;-;q,x)$ has the following representation:
\begin{align}&\left(\mathcal{B}_{q^{r-1}}^+{}_r\varphi_0(0,0,\dots ,0;-;q,*)\right)(\xi )=
e_q(\xi )\notag\\
&=\frac{(q;q)_\infty}{\theta_q(-\xi )}
\left[{}_0\varphi_{r-1}\left(-;q,q^2,\dots ,q^{r-1};q^r,(q^r)^{\frac{r-1}{2}}\left(\frac{q}{\xi}\right)^r\right)\right.\notag\\
&+\frac{(-1)}{1-q}\frac{q}{\xi}{}_0\varphi_{r-1}\left(-;q^2,\dots ,q^{r-1},q^{r+1};q^r,(q^r)^{\frac{r+1}{2}}\left(\frac{q}{\xi}\right)^r\right)\notag\\
&+\dots +\frac{(-1)^{r-1}q^{\frac{(r-1)(r-2)}{2}}}{(1-q)(1-q^2)\dots (1-q^{r-1})}\left(\frac{q}{\xi}\right)^{r-1}\notag\\
&\left.{}_0\varphi_{r-1}\left(-;q^{r+1},\dots ,q^{2(r-1)};q^r,(q^r)^{\frac{3(r-1)}{2}}\left(\frac{q}{\xi}\right)^r\right)
\right]=:\hat{\varphi}(\xi ).\notag
\end{align}
Moreover, we apply the $q$-Laplace transformation of level $r-1$ to $\hat{\varphi} (\xi)$, we obtain the following theorem.

\begin{thm}For any $\mathbb{C}^*\setminus [-\lambda ;q^{r-1}]$, we have
\begin{align}&\left(\mathcal{L}_{q^{r-1,\lambda}}^+\circ \mathcal{B}_{q^{r-1}}^+{}_r\varphi_0(0,0,\dots ,0;-;q,x)\right)\notag\\
&=\frac{1}{\theta_{q^{r-1}}\left(\frac{\lambda}{x}\right)}\frac{(q;q)_\infty}{\theta_q(-\lambda )}\theta_{q^{r(r-1)}\left((-1)^{r-1}q^{\frac{(r-1)(r-2)}{2}}\left(\frac{\lambda}{x}\right)^r\right)}\notag\\
&\times{}_{r-1}\varphi_{r-1}\left(0,0,\dots ,0;q,q^2,\dots ,q^{r-1};q^r,q^{\frac{(r-1)(r-2)}{2}}/x\right)\notag\\
&+\dots \notag\\
&+\frac{1}{\theta_{q^{r-1}}\left(\frac{\lambda}{x}\right)}\frac{(q;q)_\infty}{\theta_q(-\lambda )}\theta_{q^{r(r-1)}\left((-1)^{r-1}q^{\frac{3r(r-1)}{2}}\left(\frac{\lambda}{x}\right)^r\right)}\notag\\
&\times \frac{(-1)^{r-1}q^{\frac{(r-1)(r-2)}{2}}}{(1-q)(1-q^2)\dots (1-q^{r-1})}\left(\frac{q}{\lambda}\right)^{r-1}\notag\\
&\times{}_{r-1}\varphi_{r-1}\left(0,0,\dots ,0;q^r,q^{r+1},\dots ,q^{2r-1};q^r,q^{\frac{3r(r-1)}{2}}/x\right)\notag
\end{align}
\end{thm}
%%%%%%%%%%%%%%%%%%%%%%acknowledgement%%%%%%%%%%%%%%%%%%%%%%%%%%%%%%%%%%%

\section*{Acknowledgement}
I am deeply grateful to Professor Yousuke Ohyama whose comments and suggestions were of inestimable value for my study. This study is partially supported by Fujikin Incorporated and Association for the Promotion of Production Technology.
%%%%%%%%%%%%%%%%%%%%%%references%%%%%%%%%%%%%%%%%%%%%%%%%%%%%%%%%%%%%%%%

\end{document}